\documentclass[11pt]{article}
\usepackage{amssymb,amsmath,amsthm}
\usepackage{epsfig}
\usepackage{breqn}
\usepackage{rotating}
\usepackage{amsfonts}
\usepackage{setspace}
\usepackage{fullpage}
 \usepackage{bbold} 
\usepackage{amsmath}
\usepackage{comment}
\usepackage{pgf,tikz}
\usepackage{mathtools}  
\usepackage{hyperref}
\usepackage{authblk}
\usepackage{setspace}
\usepackage{bbm}
\newtheorem{theorem}{Theorem}
\newtheorem{claim}[theorem]{Claim}

\newtheorem{corollary}[theorem]{Corollary}
\newtheorem{definition}[theorem]{Definition}

\newtheorem{example}{Example}

\newtheorem{conjecture}{Conjecture}

\newcommand{\floor}[1]{\left\lfloor{#1}\right\rfloor}

\def\h{\mathcal H}
\def\P{\mathcal P}

\def\F{\mathcal F}

\newcommand{\abs}[1]{\left\lvert{#1}\right\rvert}

\title{Pósa-type results for Berge hypergraphs}

\author{
Nika Salia\thanks{King Fahd University of Petroleum and Minerals, Dhahran, Saudi Arabia, email: \texttt{salianika@gmail.com} } 
}

\begin{document}
\maketitle
\begin{abstract}
A Berge cycle of length $k$ in a hypergraph $\mathcal H$ is a sequence of distinct vertices and hyperedges $v_1,h_1,v_2,h_2,\dots,v_{k},h_k$  such that $v_{i},v_{i+1}\in h_i$ for all $i\in[k]$, indices taken modulo $k$. F\"uredi, Kostochka, and Luo recently gave sharp Dirac-type minimum degree conditions that force non-uniform hypergraphs to have Hamiltonian Berge cycles. 
We give a sharp Pósa-type lower bound for $r$-uniform and non-uniform hypergraphs that force Hamiltonian Berge cycles.
\end{abstract}

\section{Introduction}
The study of Hamiltonian cycles is one of the essential topics of Graph Theory. In the present paper, we study sufficient degree conditions for a hypergraph to be Hamiltonian, in both uniform and non-uniform cases. We call a hypergraph Hamiltonian if there is a Berge cycle containing all of the vertices of the hypergraph as defining vertices. Note that it is natural to follow the definition of Berge for cycles in hypergraphs since there is a one-to-one correspondence between Berge cycles of the hypergraph and cycles in the incidence bipartite graph of the hypergraph. For a hypergraph $\h$ consider the incidence bipartite graph $G(A, B)$, where the vertices in $A$ represent vertices of $\h$ and the vertices in $B$ represent hyperedges of $\h$. A vertex $a \in A$ is adjacent with a vertex $b \in B$ in $G$ if and only if the vertex in $\h$ corresponding to $a$ is contained in the hyperedge corresponding to $b$ in $\h$.  There is a one-to-one correspondence between cycles in $G$ and Berge cycles of $\h$. The corresponding cycle in $G$ of a Hamiltonian Berge cycle in $\h$ is a cycle containing all vertices of the set $A$.
In~\cite{jackson1981cycles, jackson1985long, kostochka2020super, kostochka2020conditions,kostochka2022longest} is related work on cycles covering color classes in bipartite graphs.
In the coming subsections, we present the motivation for this work and introduce some necessary definitions and notions. 

\subsection{Hamiltonian cycles in graphs}

The Hamiltonicity of graphs is a well-studied problem. In this subsection, we only state those results which are the direct motivation of this work.  For the recent developments in this topic, we refer the reader to the following survey papers  \cite{gould2003advances, gould2014recent}. 

For a graph $G$, a cycle containing every vertex of $G$ is called a \emph{Hamiltonian cycle}.  Graphs containing at least one Hamiltonian cycle are called \emph{Hamiltonian}. We also call a path \emph{Hamiltonian path} if it contains every vertex of the graph.  
Dirac~\cite{dirac1952some} in 1952 proved an important sufficient condition for a graph to be Hamiltonian. They showed that every graph $G$ on $n$ vertices with $n\geq 3$ and minimum degree $\delta(G)$ at least $\frac{n}{2}$ is Hamiltonian. This theorem is sharp in the sense that one can not replace $\frac{n}{2}$ with anything less, for example with $\floor{\frac{n}{2}}$. In 1962 Pósa proved a strengthening of this theorem. 
\begin{theorem}[Pósa~\cite{posa1962theorem}]\label{thm:posa}
Let $G$ be an $n$ vertex graph. Let $n\geq 3$ and the degree sequence of $G$ be $d_1\leq d_2\leq \dots \leq d_n$. 
If for all  $k<\frac{n}{2}$ the inequality $k<d_k$ holds then $G$ is Hamiltonian.
\end{theorem}

Several theorems give sufficient conditions for some classes of graphs to be Hamiltonian for a reference see \cite{bondy1969properties,tutte1956theorem}.  In 1972, Chvátal found the necessary and sufficient condition for an integer sequence to be Hamiltonian.
We say an \emph{integer sequence} $d_1\leq d_2\leq \dots \leq d_n$ is \emph{Hamiltonian} if every graph with $n$ vertices and the degree sequence pointwise greater than the integer sequence is Hamiltonian.  
 
 \begin{theorem}[Chvátal~\cite{chvatal1972hamilton}]\label{thm:chvatal}
 An integer sequence $d_1\leq d_2\leq \dots \leq d_n$, such that $n\geq 3$ is Hamiltonian if and only if the following holds for every $k<\frac{n}{2}$
 \[
 d_k\leq k \Rightarrow d_{n-k}\geq n-k.
 \]
 \end{theorem}
 
 In the following subsection, we introduce some notions and results for Hamiltonian Berge hypergraphs.

\subsection{Hamiltonicity for hypergraphs and main results} 

To describe the problem and existing results,  we need to introduce some standard notions. 
A hypergraph $\mathcal{H}$ is defined by a pair $(V(\mathcal{H}), E(\mathcal{H}))$, where $V(\mathcal{H})$ denotes the set of vertices and $E(\mathcal{H})$ denotes the set of hyperedges. Furthermore, $E(\mathcal{H}) \subseteq \mathcal{P}(V(\mathcal{H}))$, with $\mathcal{P}(V(\mathcal{H}))$ representing the power set of $V(\mathcal{H})$.
 We say $\h$ is $r$-uniform if every hyperedge has size~$r$. For a vertex $v\in V(\h)$, the degree of a vertex is the number of hyperedges incident with it and is denoted by $d_{\h}(v)$. If the host hypergraph is clear from the context we will use $d(v)$ instead of $d_{\h}(v)$. 
We define the neighbourhood of a vertex $v$ as $N_{\h}(v)=\{u\in V(\h)\setminus \{v\}:\{u,v\}\subseteq h, h\in \h \}$.
The closed neighborhood of $v$ is defined as $N_{\h}[v]$, where $N_{\h}[v]=N_{\h}(v)\cup \{v\}$. 
For a vertex set $S$, we denote $N_{\h}(S):=\cup_{v \in S}N(v)$ and $N_{\h}[S]:=\cup_{v \in S}N[v]$ .
For a hypergraph $\h$ and sub-hypergraph $\h'$ the hypergraph on the same vertex set as $\h$ and with hyperedges $E(\h)\setminus E(\h')$ is denoted by $\h\setminus \h'$. For a set $S$ and an integer $r$ let us denote the set of all subsets of $S$ of size $r$ by $\binom{S}{r}$. Let us introduce the notions of  Berge paths and Berge cycles~\cite{berge1973graphs}.

\begin{definition}
A \emph{Berge path} of length $t$  is an alternating sequence of $t+1$  $(t)$ distinct vertices (hyperedges) of the hypergraph,  $v_1, e_1, v_2, e_2, v_3, \dots, e_t, v_{t+1}$ such that $v_{i},v_{i+1} \in e_i$, for $i \in [t]$. The vertices $v_1, v_2, \dots, v_{t+1}$ are called \emph{defining vertices} and the hyperedges $e_1,e_2,\dots,e_t$ are called \emph{defining hyperedges} of the Berge path. 

Similarly, a \emph{Berge cycle} of length $t$ is an alternating sequence of $t$ distinct vertices and hyperedges of the hypergraph,  $v_1, e_1, v_2, e_2, v_3, \dots, v_t, e_t$, such that $v_{i},v_{i+1} \in e_i$, for $i\in [t]$, where indices are taken modulo $t$. The vertices $v_1, v_2, \dots, v_{t}$ are called \emph{defining vertices} and the hyperedges $e_1,e_2,\dots,e_t$ are called \emph{defining hyperedges} of the Berge cycle. 
\end{definition}

Long Berge cycles are well-studied for hypergraphs. Turán-type questions for uniform hypergraphs without long Berge cycles are settled in \cite{ergemlidze2020avoiding, furedi2019avoiding, gyHori2021structure, kostochka2020r}. Bermond, Germa, Heydemann, and Sotteau \cite{bermond1978hypergraphes}   found a Dirac-type condition forcing long Berge cycles for uniform hypergraphs.  Recently Coulson and  Perarnau~\cite{coulson2020rainbow}  found Dirac-type condition forcing Berge Hamiltonicity in hypergraphs.   
Füredi, Kostochka, and Luo~\cite{furedi2020berge} generalized Dirac's theorem for non-uniform hypergraphs. For uniform linear hypergraphs, Jiang and  Ma  \cite{jiang2018cycles}  settled a conjecture of Verstraëte, by finding asymptotic minimum degree condition necessary for the existence of Berge cycles of $k$ consecutive lengths.  For the study of asymptotic minimum degree thresholds that force matching see the following work~\cite{bowtell2022degree, han2009perfect}.
 Katona and Kierstead~\cite{katona1999hamiltonian} introduced an alternative definition of Hamiltonian cycles in hypergraphs, a notion that has garnered considerable attention over recent decades. A seminal contribution to this area was made by Rödl, Ruciński, and Szemerédi~\cite{rodl2008approximate}, who established a Dirac-type condition for Hamiltonian cycles under this definition see also work by Schülke~\cite{schulke2023pair},  which presents a Pósa-type condition for Hamiltonicity.
In the subsequent part of this work, we adopt the term ''Hamiltonian`` to specifically refer to the concept of Berge Hamiltonian cycles in hypergraphs.
Here we state the theorem of   Füredi, Kostochka, and Luo.

\begin{theorem}[Füredi, Kostochka, Luo~\cite{furedi2020berge}] \label{thm:F_K_L_non_unif_Hyp}
Let $n\geq 15$ and let $\h$ be an $n$-vertex hypergraph such that $\delta(\h)\geq 2^{\frac{n-1}{2}}+1$ if $n$ is odd and $\delta(\h)\geq $2 $^{\frac{n-2}{2}}+2$ if $n$ is even. Then H contains a Berge Hamiltonian cycle.
\end{theorem}

For $r$-uniform hypergraphs Kostochka, Luo, and  McCourt~\cite{kostochka2021dirac} proved exact Dirac-type bounds for Hamiltonian Berge cycles in $r$-uniform $n$-vertex hypergraphs for all $3\leq r < n$, their bounds are different for $r<n/2$ and $r\geq n/2$.  
Let us note that, the proof techniques used in works~\cite{gerbner2024stability, gyHori2019connected} also provide the same result for $r<n/2$.

We say an integer sequence $(d_1,d_2,\ldots,d_n)$  is  $r$-Hamiltonian if every $r$-uniform hypergraph with the degree sequence pointwise greater than the integer sequence is Hamiltonian. 
We say an integer sequence $(d_1,d_2,\ldots,d_n)$  is  $\mathbb N$-Hamiltonian if every non-uniform hypergraph with the degree sequence pointwise greater than the integer sequence is Hamiltonian.

\begin{theorem}\label{Theorem_Posa_r_Uniform}
An integer sequence $(d_1,d_2,\ldots,d_n)$ such that $d_1\leq d_2\leq \dots\leq d_n$, $n>2r$ and $r\geq 3$ is  $r$-Hamiltonian if the following conditions hold
\begin{equation}\label{Condition:Posa1}
   d_i>i \mbox{ for } 1 \leq i < r,
\end{equation}

\begin{equation}\label{Condition:Posa2}
  d_i> \binom{i}{r-1}  \mbox{ for } r \leq i \leq \floor{\frac{n-1}{2}},
\end{equation}

\begin{equation}\label{Condition:Posa3}
d_{\frac{n-2}{2}}> \binom{\frac{n-2}{2}}{r-1} +1  \mbox{ if $n$ is even.} 
\end{equation}

\end{theorem} 
One may interpret  Theorem~\ref{Theorem_Posa_r_Uniform} as an analog of Theorem~\ref{thm:posa} for $r$-uniform hypergraphs. Besides we would like to prove the analog of Theorem~\ref{thm:posa} for non-uniform hypergraphs as a strengthening of Theorem~\ref{thm:F_K_L_non_unif_Hyp}.

\begin{theorem}\label{Theorem_Posa_non_Uniform}
An integer sequence $(d_1,d_2,\ldots,d_n)$ such that $d_1\leq d_2\leq \dots\leq d_n$ and $n>40$ is $\mathbb N$-Hamiltonian if the following conditions hold

\begin{equation}\label{Condition:PosaNON-uniform_1}
  d_i> 2^i  \mbox{ for } 1 \leq i \leq \floor{\frac{n-1}{2}},
\end{equation}

\begin{equation}\label{Condition:PosaNON-uniform_2}
d_{\frac{n-2}{2}}> 2^{\frac{n-2}{2}}+1  \mbox{ if $n$ is even.} 
\end{equation}

\end{theorem} 

In the following subsection, we show that the conditions of  Theorem~\ref{Theorem_Posa_r_Uniform} and Theorem~\ref{Theorem_Posa_non_Uniform} are sharp.

\subsection{Examples showing the sharpness of conditions}
In this subsection, we show that it is impossible to strengthen  Theorem~\ref{Theorem_Posa_r_Uniform} or Theorem~\ref{Theorem_Posa_non_Uniform}  by changing a condition for a given $i$.  We start with Theorem~\ref{Theorem_Posa_r_Uniform} and demonstrate that it is impossible to strengthen it by modifying a condition for some fixed $i=k$, where $1\leq k\leq \floor{\frac{n-1}{2}}$.   

Example~\ref{Example:Degree_Sequence_r_Uniform_k<r} shows the sharpness of Condition~\ref{Condition:Posa1} for all $k$, where $1\leq k <r$. The idea of this construction is to construct a hypergraph with a special vertex set of size $k$ incident with $k$ hyperedges only, therefore no Hamiltonian cycle.

\begin{example}\label{Example:Degree_Sequence_r_Uniform_k<r}
For integers $n$, $r$, and $k$, with the conditions $n > 2r > 2k > 0$, we define $\mathcal{H}^1_k$ as an $n$-vertex, $r$-uniform hypergraph as follows. The vertex set of $\mathcal{H}^1_k$ is partitioned into two disjoint sets, $V_1$ and $V_2$, where $|V_1| = k$ and $|V_2| = n - k$. The hyperedge set of $\mathcal{H}^1_k$ comprises all hyperedges from $\binom{V_2}{r}$ and includes $k$ distinct hyperedges, each of which contains $V_1$ as a proper subset.
\end{example} 

Let  $d_1\leq d_2\leq \dots\leq d_n$ be the degree sequence of $\h^1_k$. We have $d_1=d_2=\dots=d_k=k$ and  $\binom{n-k-1}{r-1}\leq d_{k+1}\leq d_{k+2} \leq \dots \leq d_{n}$. Observe that, for the degree sequence of hypergraph $\h^1_k$ all conditions of Theorem~\ref{Theorem_Posa_r_Uniform} hold except one, Condition~\ref{Condition:Posa1} for $i=k$.  In particular $d_k=k$ instead of $d_k>k$. 
Clearly, $\h^1_k$ is non-Hamiltonian since the vertices in $V_1$ are incident with only $k$ hyperedges, therefore, there is no Berge cycle in  $\h^1_k$ that contains all vertices of $V_1$ and is longer than $\abs{V_1}=k$.

Example~\ref{Example:Degree_Sequence_r_Uniform_k>r} shows the sharpness of Condition~\ref{Condition:Posa2} for all $k$, $r \leq k \leq \floor{\frac{n-1}{2}}$.  The idea behind this construction is to construct a hypergraph with a special vertex set of size $k$, each adjacent to only $k$ vertices.

\begin{example}\label{Example:Degree_Sequence_r_Uniform_k>r}
For integers $n,r$ and $k$, such that   $3\leq r \leq  k < \frac{n}{2}$, let $\h^2_k$ be an $n$-vertex, $r$-uniform hypergraph. Let us partition the vertex set of $\h^2_k$, into three disjoint sets $V_1,V_2$ and $V_3$ of sizes $\abs{V_1}=k$, $\abs{V_2}=k$ and $\abs{V_3}=n-2k$. The hyperedges of $\h^2_k$ are
\[
E(\h^2_k)=\left\{h\in \binom{V(\h^2_k)}{r}:  \left( h\subset V_1\cup V_2  \mbox{ and } \abs{h \cap V_1}=1 \right) \mbox{ or }  h \in V_2\cup V_3   \right\}.
\]
\end{example}
The degree sequence of $\h^2_k$ is $d_1=d_2=\dots=d_k=\binom{k}{r-1}$, $d_{k+1}=\dots=d_{n-k}=\binom{n-k-1}{r-1}$ and $d_{n-k+1}=\cdots=d_{n}=\binom{n-k-1}{r-1}+k\binom{k-1}{r-2}$.
Observe that, for the degree sequence of hypergraph $\h^2_k$ all conditions of Theorem~\ref{Theorem_Posa_r_Uniform} hold except one, Condition~\ref{Condition:Posa2} for $i=k$. 
In particular $d_k=\binom{k}{r-1}$ instead of $d_k>\binom{k}{r-1}$. 
Note that vertices of $V_1$ are pairwise non-adjacent, and the number of vertices adjacent to $V_1$ is $k=\abs{V_1}$.
Therefore, there is no Berge cycle in  $\h^2_k$ containing all vertices of $V_1$ longer than $2\abs{V_1}=2k$, hence $\h^2_k$  is not Hamiltonian since $n>2k$.  

The next example shows the sharpness of Condition~\ref{Condition:Posa3}.  The idea is very similar to the previous example.

\begin{example}\label{Example:Degree_Sequence_r_Uniform_n/2}
For integers $n$ and $r$, such that  $2|n$, $3\leq r < \frac{n}{2}$, let $\h^3$ be an $n$ vertex, $r$-uniform hypergraph. Let us partition the vertex set of $\h^3$, into two disjoint sets $V_1$ and $V_2$ of sizes $\abs{V_1}=\frac{n}{2}+1$ and $\abs{V_2}=\frac{n}{2}-1$. Let us fix a subset of $V_1$ of size $r$ and denote it by $h'$. The hyperedge set of $\h^3$ is 
\[
E(\h^3)=\left\{h\in \binom{V(\h^3)}{r}:  (h\subset V_1\cup V_2  \mbox{ and } \abs{h \cap V_1}\leq 1) \mbox{ or }  h=h'   \right\}.
\]
\end{example}
The degree sequence of $\h^3$ is $d_1=d_2=\dots=d_{\frac{n}{2}+1-r}=\binom{\frac{n}{2}-1}{r-1}$, $d_{\frac{n}{2}+2-r}=\dots=d_{\frac{n}{2}+1}=\binom{\frac{n}{2}-1}{r-1}+1$ and $d_{\frac{n}{2}+2}=\cdots=d_{n}=\binom{\frac{n}{2}-2}{r-1}+(\frac{n}{2}+1)\binom{\frac{n}{2}-2}{r-2}$.
As one can observe all conditions of Theorem~\ref{Theorem_Posa_r_Uniform} hold but Condition~\ref{Condition:Posa3}. The number of vertices in $V_1$  is $\frac{n}{2}+1$, therefore, if there is a Hamiltonian Berge cycle then there should be at least two pairs of consecutive vertices of $V_1$ on the cycle. 
This is not possible since the number of hyperedges incident with at least two vertices of $V_1$ is just one.

These three examples show the sharpness of each condition of Theorem~\ref{Theorem_Posa_r_Uniform}, since examples showing the sharpness of conditions of Theorem~\ref{Theorem_Posa_non_Uniform} are very similar we will omit them in this manuscript.

\section{Proofs}
In this section, we prove Theorem~\ref{Theorem_Posa_r_Uniform} and Theorem~\ref{Theorem_Posa_non_Uniform}, since the methods used to prove the theorems are similar, we start the proof of both theorems together, describe common tools, and in the end, we split the proof into two subsections. 

Let $\F$ be a hypergraph with $V(\F):=\{u_1,u_2,u_3,\dots, u_n\}$, and with a Hamiltonian Berge path
\[P:=u_1,f_1,u_2,f_2,u_3,\dots,f_{n-1},u_n.\]
Where $f_1,f_2,\dots,f_{n-1}$
are distinct hyperedges of $\F$. We use the following three ways to permute the vertices and produce alternative Hamiltonian Berge paths from $P$. 
\begin{itemize}

    \item \textbf{Permutation~$1$- with a defining hyperedge $f_i$.} \\
    If $u_1\in f_i$ then $P'$ is a Hamiltonian Berge path where
\begin{equation*}
    P':=u_{i},f_{i-1},u_{i-1},\dots,u_1,f_i,u_{i+1},f_{i+1},\dots,f_{n-1},u_n.
\end{equation*}

    \item \textbf{Permutation~$2$- with a non-defining hyperedge $f$.} \\
    If $\{u_1,u_{i+1}\}\subset f$, $f\notin\{f_1,f_2,\dots,f_{n-1}\}$ then $P'$ is a Hamiltonian Berge path where
   \begin{equation*}
    P':=u_{i},f_{i-1},u_{i-1},\dots,u_1,f,u_{i+1},f_{i+1},\dots,f_{n-1},u_n.
   \end{equation*}
    
    \item \textbf{Permutation~$3$-  with a defining hyperedge $f_i$ and a non-defining hyperedge $f$.} \\ If for some integers  $i$ and $j$, with $i>j$, if we have  $u_1\in f_i$ and for some hyperedge $f$, $f\notin\{f_1,f_2,\dots,f_{n-1}\}$, we have $u_j,u_{i+1} \in f$ then $P'$ is a Hamiltonian Berge path where

\begin{equation*}
P':=u_{j+1},f_{j+1}, \ldots ,u_{i},f_i, u_1, f_1, u_2 \ldots u_j, f, u_{i+1}, f_{i+1}\ldots,f_{n-1}, u_{n}.
\end{equation*}
\end{itemize}

We prove Theorem~\ref{Theorem_Posa_r_Uniform} and Theorem~\ref{Theorem_Posa_non_Uniform} by assuming a contradiction, let $\h$ be an $r$-uniform / $\mathbb N$-uniform hypergraph satisfying conditions of the corresponding theorem but containing no Hamiltonian Berge cycle.  
Without loss of generality, we may assume that $\h$ is maximal, without containing a Hamiltonian cycle, in the sense that adding any hyperedge to $\h$ creates a Hamiltonian hypergraph. 
Therefore, the longest Berge path in $\h$ is Hamiltonian. 
Let $P$ be a longest Berge path 
\[
P:= v_1, h_1, v_2, h_2\dots, h_{n-1}, v_{n}.
\]
Without loss of generality, we assume $d(v_1)\leq d(v_n)$ and that $d(v_1)+d(v_n)$ is  maximal among all Hamiltonian Berge paths contained in $\h$. Subject to this, let us assume among all such paths $P$ maximizes $\abs{N_{\h\setminus P}(v_1)}+\abs{N_{\h\setminus P}(v_n)}$.

Let $x$ be the largest integer for which $\{v_1,v_2,v_3,\dots,v_x\}\subseteq N_{\h \setminus P}[v_{1}]$. We denote $S_1:=\{v_1\}\cup\{v_2,v_3,\dots,v_{x-1}\}$. 
For each vertex $v_i\in \{v_2,v_3,\dots,v_x\}$, let $P_{v_i}$ be a Berge path obtained by altering $P$ with Permutation~$2$ using some non-defining hyperedge incident with $v_i$ and $v_1$.
Let us denote the family of these Hamiltonian Berge paths including $P$ with $\P$.  
There are $x$ Hamiltonian Berge paths in $\P$, in particular for each vertex $v_j$ of $S_1$ there is a  Hamiltonian Berge path $P_{v_j}$ in $\P$, starting at $v_j$ and finishing at $v_n$. Note that for each $h_i$, $1\leq i\leq x-1$, there is a  Hamiltonian Berge path $P_{v_{i}}$ in $\P$ not using $h_i$ as a defining hyperedge and the terminal vertex $v_i$ is incident to~$h_i$.   
We denote 
\[
\h_1:=(\h\setminus P) \cup \{h_1,\dots,h_{x-1}\}.  
\]
Let $T_1$ be the set of terminal vertices of Hamiltonian paths whose other terminal vertex is $v_n$.
Note that $v_1 \in T_1$ and $v_1$ has the maximum degree among all the vertices of $T_1$, therefore, we have $d(v_1)>min\left\{\abs{T_1},\floor{\frac{n-1}{2}} \right\}$ from the conditions of Theorem~\ref{Theorem_Posa_r_Uniform} and  Theorem~\ref{Theorem_Posa_non_Uniform}. 
Let us denote the number of defining hyperedges $h_i$ incident with $v_1$  with $k$.  For a set $A$, $A\subseteq \{v_1,v_2,\dots,v_n\}$, we define the left shift of $A$ as
\[
A^-:=\{v_i:v_{i+1}\in A, i\geq 1\}.
\]
Note that if $v_1\in A$ then $A^-$ contains one less vertex than $A$. 
The right shift of set A is defined analogously and it is denoted by $A^+$.

\begin{claim}\label{Claim:basic}
We have $ N_{\h_1}[S_{1}]^-\cup  \{v_{y}| v_{i}\in h_y,i<x\leq y\} \subseteq T_{1}$.
\end{claim}
\begin{proof}
For each vertex $v_i\in S_1$ there is a path $P_{v_i}\in \P$ such that $v_i$ is a terminal vertex of $P_{v_i}$, where $P_{v_i}$ is a Hamiltonian path obtained from $P$ after altering it with Permutation~$2$. 
Now consider  a Hamiltonian Berge path obtained from $\P_{v_i}$ for $v_i \in S_1$ and alter it with Permutation~$2$ for each non-defining hyperedge incident with $v_i$ thus we get  $N_{\h\setminus \{h_i\}}[v_i]^-\subseteq T_1$. 
By altering $\P_{v_i}$ with Permutation~$1$ for each defining hyperedge incident with $v_i$, we get  $  \{v_{y}| v_{i}\in h_y,y\geq x\}\subseteq T_1$. 
Thus we have 
\[
N_{\h_1}[S_1]^-\cup  \{v_{y}| v_{i}\in h_y,i<x\leq y\}=\bigcup_{v_i\in S_1} N_{\h \setminus P_{v_i}}[v_i]^-\cup  \{v_{y}| v_{i}\in h_y,x \leq y\}.
\] 

\end{proof}

From here we break the proofs of Theorem~\ref{Theorem_Posa_r_Uniform} and Theorem~\ref{Theorem_Posa_non_Uniform}. In the following sub-section,  we prove Theorem~\ref{Theorem_Posa_r_Uniform} and then we prove  Theorem~\ref{Theorem_Posa_non_Uniform}.

\subsection{Proof of the theorem for uniform hypergraphs}

\begin{proof}[Proof of Theorem~\ref{Theorem_Posa_r_Uniform}]

With the sequence of the following claims we are going to prove $d(v_n)\geq d(v_1)>
\binom{\floor{\frac{n-1}{2}}}{r-1}+\mathbbm{1}_{2|n}$. At first we show that $d(v_1)>\binom{r+1}{r-1}$ with the following claim.

\begin{claim}\label{Claim_Size_Degree_v}
The size of $T_1$ is at least  $r+1$.
\end{claim}
\begin{proof}
We prove the claim in three cases depending on how many non-defining hyperedges are incident with $v_1$. 

\textbf{Case $1$.}
The vertex $v_1$ is incident with at least two non-defining hyperedges $e_1$ and $e_2$, none of which is incident with $v_2$ then we have $\{v_1\} \cup ( e_1\cup e_2)^- \subseteq T_1$, by Claim~\ref{Claim:basic}. We have  $\abs{e_1\cup e_2}\geq r+1$ and $v_2\notin e_1\cup e_2$. Hence we have  $\abs{T_1}\geq r+1$. 

\textbf{Case 2.} The vertex $v_1$ is incident with at least two non-defining hyperedges $e_1$ and $e_2$, but at least one of them is incident with $v_2$. Without loss of generality, we may assume $e_1$ is incident with $v_2$. Then we may replace $h_1$ with $e_1$ in $P$, without changing defining vertices of $P$. Hence we have $(e_1\cup e_2 \cup h_1)^-\subseteq T_1$ by Claim~\ref{Claim:basic}. We have  $\abs{(e_1\cup e_2 \cup h_1)^-}\geq r$ hence either we are done or we may suppose to the contrary that $\abs{T_1}=\abs{(e_1\cup e_2 \cup h_1)^-}=r$. Note that in this case $e_1, e_2$ and $h_1$ are distinct hyperedges which are incident with $v_1$ and all but one vertex from $(e_1\cup e_2 \cup h_1)$. In particular, for any two distinct vertices from   $e_1\cup e_2 \cup h_1$ there exists a hyperedge  $e_1, e_2$ or $h_1$ incident with both of the vertices.

We have $d(v_1)>\abs{T_1}= r$ therefore, there exists a hyperedge $h'$ incident with $v_1$ which is not a subset of $e_1\cup e_2 \cup h_1$. 
If $h'$ is not a defining hyperedge, then 
$(h')^- \subset T_1$ by  Claim~\ref{Claim:basic}, 
which is a contradiction since $(e_1\cup e_2 \cup h_1)^-=T_1$ and $(h')^- \not\subseteq (e_1\cup e_2 \cup h_1)^-$. 
Hence $h'$ is a defining hyperedge $h_i$ for some fixed $i$.
We have $v_{i}\in T_1$ by Claim~\ref{Claim:basic}. 
Since $T_1=(e_1\cup e_2 \cup h_1)^-$ we have $v_{i+1} \in e_1\cup e_2 \cup h_1$. In the following part of the proof we  show that $\{v_1,v_2,\ldots, v_{i+1}\} \subseteq e_1\cup e_2 \cup h_1$. 

We have $\{v_1,v_2,v_{i+1} \} \subseteq e_1\cup e_2 \cup h_1$.
Let $\gamma$ with $i+1>\gamma\geq 2$, be the minimum integer such that $v_{\gamma} \notin T_1$ if it exists. We have $v_{\gamma-1}  \in e_1\cup e_2 \cup h_1$.  We  have shown that one of the hyperedges $e_1, e_2$ or $h_1$ is incident with $v_{\gamma-1}$ and $v_{i+1}$. 
If such hyperedge is $h_1$, we may exchange $h_1$ with $e_1$, in $P$, and the hyperedge incident with   $v_{\gamma-1}$ and $v_{i+1}$ will be non-defining. Thus, we may assume without loss of generality that it is $e_j$ for some $j\in [2]$. After altering $P$ with Permutation~$3$ for defining hyperedge $h_i$ and non-defining hyperedge $e_j$,  we get  $v_{\gamma}\in T_1$. A contradiction to $v_{\gamma} \notin T_1$.

Finally we have $\{v_i,v_{i+1}\} \subseteq e_1\cup e_2 \cup h_1$, we  have shown that one of the hyperedges $e_1, e_2$ or $h_1$ is incident with $v_{i}$ and $v_{i+1}$. We may assume without loss of generality that it is $e_2$. Hence we may replace $h_i$ with  $e_2$. By Claim~\ref{Claim:basic} we have $h_i^-\subset T_1$. Since $h_i \not\subset e_1\cup e_2 \cup h_1$ we have $\abs{T_1}>r$, a contradiction. 

\textbf{Case $3$.} The vertex $v_1$ is not incident with at least two non-defining hyperedges. 
If vertex $v_1$ is not incident with any non-defining hyperedge then, for each defining hyperedge $h_i$ incident with $v_1$, $v_i \in T_1$ by Claim~\ref{Claim:basic} and we are done. Thus we may assume the vertex $v_1$ is incident with exactly one non-defining hyperedge.

Let $h$ be a $P$ non-defining hyperedge incident with $v_1$. Then we have $h^-\subset T_{1}$. If $v_2 \in h$, then we may replace $h_1$ with $h$ and apply the same for the replaced path and $h_1$. Hence we have $(h_1\cup h)^-\in T_1$. If $v_2 \notin h$, then $(\{v_2\}\cup h)^-\in T_1$. In both cases, we have $\abs{T_1}\geq r$. Either we are done or we have $\abs{T_1}=r$ and $d(v_1)\geq r+1$.

Let $h_{x_1}, h_{x_2},\ldots, h_{x_r}$ be $P$ defining hyperedges incident with $v_1$, for some $1=x_1<x_2<\ldots<x_r\leq n$. Since $\abs{T_1}=r$, we have $T_1=\{v_{x_1},v_{x_2}, \ldots, v_{x_r}\}$ and $h\subseteq \{v_1, v_{x_1+1}, v_{x_2+1}, \ldots, v_{x_r+1}\}$. We have $v_2\in T_1$, since we can remove $v_1,h_1$ from $P$ and replace  $v_{x_i}, h_{x_i}, v_{x_i+1}$ by $v_{x_i}, h_{x_i}, v_1, h, v_{x_i+1}$ for some $i$, $v_{x_i+1}\in h$, i.e. altering $P$ with Permutation~$3$. Hence we have $x_2=2$. 
We have $h\subseteq \{v_1, v_{2}, v_{3},v_{x_3+1}, \ldots, v_{x_r+1}\}$, in particular, $h$ is a hyperedge incident with all but one vertex from the set $\{v_1, v_{2}, v_{3},v_{x_3+1}, \ldots, v_{x_r+1}\}$. We may assume $v_3\in h$, otherwise  $v_2\in h$ and we may  replace $h_1$ with $h$. Similarly as for $h$ we have $h_1\subseteq \{v_1, v_{2}, v_{3},v_{x_3+1}, \ldots, v_{x_r+1}\}$ and $h_1$ is different from $h$. Hence we have $v_3 \in h_1$.  If we replace the beginning of the path $P$,  $v_1,h_1,v_2,h_2,v_3$ with $v_2,h_1,v_1,h,v_3$, we deduce $h_2 \subseteq  \{v_1, v_{2}, v_3, v_{x_3+1}, \ldots, v_{x_r+1}\}$. Either $h$ or $h_2$ is incident with $v_{x_r+1}$, let $h'$ be the one, then the following is a Hamiltonian Berge path
\[
v_{3}, h_{3}, \ldots, v_{x_r}, h_{x_r}, v_1, h_1, v_2, h', v_{x_r+1}, \ldots, v_n.
\]
 Thus we have $v_3\in T_1$. 
Similarly as for $h_2$ we have $h_3 \subseteq  \{v_1, v_{2}, v_{3},v_4, v_{x_{5}+1} \ldots, v_{x_r+1}\}$. 

If $v_2\in h$, we may replace $h_1$ with $h$ in $P$. Either $h_1$ or $h$  contains $v_3$, so we may replace $h_2$ with that hyperedge in $P$.   Otherwise if $v_2\notin h$ then $\{v_3,v_4\}\in h$ and we may replace $h_3$ with $h$ in $P$. The hyperedge $h_3$ contains $v_2$ hence we may replace $h_1$ with it. 
Therefore, we have three different candidates $h, h_1, h_2$ for a non-defining hyperedge. 
Hence given any two distinct vertices $v',v''\in \{v_{2}, v_{3},v_4, v_{x_4+1} \ldots, v_{x_r+1}\}$,  we may assume $v',v''\in h$, without loss of generality. 
Hence, altering the path $P$ with Permutation~$3$ implies that $x_i=i$ for all possible $i$.
We have a contradiction, since every hyperedge $h,h_1,h_2,\dots, h_r$ is a distinct subset of $v_1,v_2,\dots,v_{r+1}$ and is incident to $v_1$.  Hence we have  $\abs{T_1}\geq r+1$ and we are done. 

\end{proof}

By Claim~\ref{Claim_Size_Degree_v} we have $d(v_1)> \binom{r+1}{r-1}$.  Therefore, there exists an integer $t\geq r+1$,  for which 
\[
\binom{t}{r-1}<d(v_1)\leq \binom{t+1}{r-1}.
\]
Recall the number of defining hyperedges incident with $v_1$ is denoted by $k$. Let  $\mathbbm{1}_{r\geq 4}$ be a function which equals to one for all $r\geq 4$ and zero if $r=3$.

\begin{claim}\label{claim:neighborhood_size_of_v_1}
Let $k\leq \floor{\frac{n-1}{2}}+\mathbbm{1}_{r\geq 4}$, then we have $\abs{N_{\h\setminus P}(v_1)}\geq t$ or $\abs{N_{\h\setminus P}(v_1)}= t-1$, $k=t$ and $r=3$.

\end{claim}
\begin{proof}
By altering $P$ with Permutation~$1$ for each defining hyperedge incident with $v_1$, we get distinct terminal vertices for each hyperedge.
Thus we have $\abs{T_1}\geq k$. From the conditions of Theorem~\ref{Theorem_Posa_r_Uniform}, since $v_1$ has maximum degree among all vertices of $T_1$, we have   $k \leq t+\mathbbm{1}_{r\geq 4}$.

The vertex $v_1$ is incident with $k$ defining hyperedges, and the rest of the hyperedges which are incident with $v_1$ are incident to $r-1$ other vertices in $N_{\h\setminus P}(v_1)$,  by Claim~\ref{Claim:basic}. 
Thus, we may upper bound the number  of $P$ non-defining  hyperedges incident with the vertex $v_1$ with 
$\binom{\abs{N_{\h\setminus P}(v_1)}}{r-1}$. Hence we have
\begin{equation}\label{Equation:neighborhood_size_of_v_1}
\binom{\abs{N_{\h\setminus P}(v_1)}}{r-1}+k\geq d(v_1)>  \binom{t}{r-1}=\binom{t-1}{r-1}+\binom{t-1}{r-2}.
\end{equation}
 If $r\geq 4$, then  $k \leq t+1\leq \binom{t-1}{2}\leq  \binom{t-1}{r-2}$ since $t\geq r+1\geq 5$  which implies together with Equation~\ref{Equation:neighborhood_size_of_v_1}  that $\abs{N_{\h\setminus P}(v_1)}\geq t$.  If $r=3$ and $k \leq t-1$, then  from  Equation~\ref{Equation:neighborhood_size_of_v_1} we have  $\abs{N_{\h\setminus P}(v_1)}\geq t$. Otherwise we have  $\abs{N_{\h\setminus P}(v_1)}= t-1$, $k=t$ and $r=3$.
\end{proof}

Recall that  $x$ denotes the largest integer for which $\{v_1,v_2,\dots,v_x\}\subseteq N_{\h \setminus P}[v_{1}]$.
 

\begin{claim}\label{Claim:T_1_size}
Let  $\abs{N_{\h_1}[S_1]}=\abs{N_{\h\setminus P}[v_1]}=t+1$.  Then $v_x\in T_1$ or $v_y \in T_1$, for some $y$ with $v_{y+1} \notin N_{\h_1}[S_1]$ and $v_1\in h_y$.

\end{claim}

\begin{proof}

There is a defining hyperedge $h_y$ incident with $v_1$, $y\geq x$ since $d(v_1)>\binom{t}{r-1}$ and  $\abs{N_{\h_1}(v_1)}=t$. Note that if $t\geq \floor{\frac{n-1}{2}}$ and $n$ are even then there are at least two such hyperedges and we may choose one with the minimum index. Hence either $v_{y+1} \notin N_{\h_1}[S_1] $ and we are done or $v_{y+1} \in N_{\h_1}[S_1]$ and we need to show that $v_x \in T_1$. 

If $x=1$, then there is nothing to prove. If $x=2$, then alter the path $P$ with Permutation~$3$, for defining hyperedge $h_y$ and non-defining hyperedge $h$ incident with $v_1$ and $v_{y+1}$. In particular it  removes $v_1,h_1$ from $P$ and replaces $h_y$ with $h_y, v_1,h$. Hence $v_2 \in T_1$ and $x=2$ case is also settled. 

From here we settle the final case when $x\geq 3$. 
If there is a $P$ non-defining hyperedge $f$ incident with $v_{x-1}$ and $v_{y+1}$, then altering $P$ with Permutation~$3$ for defining hyperedge $h_y$ and non-defining hyperedge $f$, we get $v_x\in T_1$, and we are done. 
Even more if $h_{x-1}$ is incident with $v_{x-1}$ and $v_{y+1}$ with the same way we get $v_x\in T_1$ and we are done. Thus we may assume that none of the $P$ non-defining hyperedges (and $h_{x-1}$) are incident with $v_{x-1}$ and $v_{y+1}$.
Note that, since $v_2$ is incident with $v_1$ with a non-defining hyperedge, hyperedge $h_1$ can also be treated as a non-defining hyperedge. Thus either we are done  $h_{1}$ is not incident with both $v_{x-1}$ and $v_{y+1}$.
Recall $N_{\h\setminus P}(v_1)=N_{\h_1}[S_1]\setminus \{v_1\}=t$, thus each of $h_1,h_2\dots, h_{x-1}$ is a subset of $N_{\h\setminus P}[v_1]$.
Hence all non-defining hyperedges incident with $v_1$, and the hyperedges  $h_1$ and  $h_{x-1}$ are subsets of $N_{\h\setminus P}[v_1]$. Even more, none of them is incident with both  $v_{x-1}$ and $v_{y+1}$. Hence the number of them is upper bounded by $\binom{\abs{T_1}}{r-1}- \binom{\abs{T_1}-2}{r-3}$. The number of defining hyperedges incident with $v_1$ is upper bounded by $t$. Note that $h_1$ and $h_{x-1}$ are double counted. Finally if $r\geq 4$ we have a contradiction 
\[
d(v_1)\leq t + \binom{t}{r-1}- \binom{t-2}{r-3}-2\leq \binom{t}{r-1}.
\]
In case $r=3$ the maximum number of hyperedges incident with $v_1$ is $\binom{\abs{N_{\h\setminus P}(v_1)}}{r-1}$. 
We have already seen that all the non-defining hyperedges and all $h_j$, $j<x$ are subsets of $N_{\h \setminus P}[v_1]$. 
For each defining hyperedge $h_j$, $j\geq x$, if there is a hyperedge $\{v_1,v_{x-1},v_{j-1}\}$ then with Permutation~$3$, $v_x$ is a terminal vertex. Otherwise if there is no such hyperedge, $\{v_1,v_{x-1},v_{j-1}\}$ for each $j$ then we have $d(v_1)\leq \binom{\abs{N_{\h\setminus P}(v_1)}}{r-1}=\binom{t}{r-1}$, a contradiction.
\end{proof}

By Claim~\ref{claim:neighborhood_size_of_v_1} and Claim~\ref{Claim:T_1_size} it is straightforward to deduce the following corollary.

\begin{corollary}\label{cor:T_1_r>3}
    If $k\leq \floor{\frac{n-1}{2}}+1$ and $r>3$, then $\abs{T_1}>t$.
\end{corollary}

\begin{claim}\label{Claim:T_1_size_3-uniform}
Let $r=3$, $k=t \leq \frac{n-1}{2}$ and $\abs{N_{\h_1}[S_1]}=t$. Then $v_{\gamma}\in T_1$, $v_1\notin h_{\gamma}$ and $v_{\gamma+1} \notin N_{\h\setminus P}(v_1)$ for some $\gamma \in \{4,5\}$. 
\end{claim}

\begin{proof}
 Recall the degree condition of $v_1$, $d(v_1)\geq \binom{t}{2}+1= \binom{t-1}{2}+t$. On the other hand $d(v_1)\leq \binom{t-1}{2}+t$, where $\binom{t-1}{2}$ bounds the non-defining hyperedges incident with $v_1$ and $t$ bounds defining hyperedges incident with $v_1$. 
 Therefore, $d(v_1)= \binom{t-1}{2}+t$ and all defining hyperedges incident with $v_1$ are not subsets of $N_{\h\setminus P}[v_1]$ and for each pair of vertices in $N_{\h\setminus P}(v_1)$ there is a non-defining hyperedge with $v_1$. 
 By Claim~\ref{Claim:basic} and the statement of Theorem~\ref{Theorem_Posa_r_Uniform} we have $T_1=\{v_i:v_1\in h_i\}$ thus  $N_{\h\setminus P}[v_1]= \{v_{i+1}:v_1\in h_i\}\setminus {u}$, for some vertex $u$.   
 Note that $v_2\notin N_{\h\setminus P}(v_1)$, otherwise by replacing $h_1$ with a non-defining hyperedge  incident with $v_1$ and $v_2$, resulting in another Hamiltonian Berge path $P'$ with the same terminal vertices. Since $h_1$ is not subset of  $N_{\h\setminus P}[v_1]$ we have $\abs{N_{\h\setminus P'}(v_1)}= t$, a contradiction to the assumption that $P$ maximizes $\abs{N_{\h\setminus P}(v_1)}+\abs{N_{\h\setminus P}(v_n)}$. 
 Hence we have $N_{\h\setminus P}[v_1]= \{v_i:v_1\in h_i\}\setminus {v_2}$. 
By altering $P$ with Permutation~$3$ for vertex $v_1$ and any defining hyperedge  incident with $v_1$ distinct from $h_1$ we get $v_2\in T_1$ thus $v_{1}\in h_2$ and $v_3\in N_{\h\setminus P}[v_1]$. 

Let $j$ be the minimal integer such that $j>3$ and $v_1 \in h_j$. Thus $v_3,v_{j+1}\in N_{\h\setminus P}[v_1]$ and there is a non-defining hyperedge $\{v_1,v_3,v_{j+1}\}$.
Then by altering $P$ with Permutation~$3$ for the hyperedge $h_j$ and $\{v_1,v_3,v_{j+1}\}$ the vertex $v_4\in T_1$. If $j\neq 4$, then we are done. If $j=4$, then $h_5$ is not incident with the vertex $v_1$ and similarly as for $v_4$, the vertex  $v_5$ is in $T_1$ thus we are done.
\end{proof}

It is straightforward to conclude the following corollary.
\begin{corollary}\label{Cor:T_1}
Let $r=3$, $k=t \leq \frac{n-1}{2}$ and $\abs{N_{\h_1}[S_1]}=t$ then $\abs{T_1}>t$.
\end{corollary}

We will show that $t\geq \floor{\frac{n-1}{2}}$. 
First, suppose that $k\geq \floor{\frac{n-1}{2}}$. 
For each defining hyperedge incident to $v_1$ there is a distinct vertex in $T_1$ by Claim~\ref{Claim:basic}. 
We have $\abs{T_1}\geq k\geq \floor{\frac{n-1}{2}}$ by Condition~\ref{Condition:Posa2} of Theorem~\ref{Theorem_Posa_r_Uniform} and the assumption that $v_1$ has the maximum degree among the vertices in $T_1$ we have $t\geq \floor{\frac{n-1}{2}}$.  

If $k< \floor{\frac{n-1}{2}}$, by Corollary~\ref{cor:T_1_r>3} and Corollary~\ref{Cor:T_1} we have $T_1\geq t+1$.
Since $v_1$ has the maximum degree in $T_1$ we get $t\geq \floor{\frac{n-1}{2}}$ by Condition~\ref{Condition:Posa2} of Theorem~\ref{Theorem_Posa_r_Uniform}. Hence we have 
\[
d(v_n)\geq d(v_1)>
\binom{\floor{\frac{n-1}{2}}}{r-1}+\mathbbm{1}_{2|n}.
\]

Let us assume $r\geq 4$. For vertex $v_n$ we  define $S_n:=\{v_n\}\cup\{v_{n},v_{n-1},v_{n-2},\dots,v_{x'+1}\}$. Where $x'$ is the smallest integer such that $\{v_{n},v_{n-1},v_{n-2},\dots,v_{x'}\}\subseteq N_{\h \setminus P}[v_{n}]$. It is simple to observe that  $N_{\h\setminus P}(v_1)^-\cap N_{\h\setminus P}(v_n)$ is empty otherwise, if we have  $v_i\in N_{\h\setminus P}(v_1)^-\cap N_{\h\setminus P}(v_n)$ then we have a Hamiltonian Berge cycle. Indeed, if $v_i$ is in $N_{\h\setminus P}(v_n)$ it means there is a $P$ non-defining hyperedge $f_n$ incident with $v_i$ and $v_n$. Similarly, there is a non-defining hyperedge $f_1$ incident with $v_{i+1} $ and $v_1$. Note that $v_n$ is not incident with $v_1$ with a $P$ non-defining hyperedge, otherwise, we would have a Hamiltonian Berge cycle therefore, $f_1$ and $f_n$ are distinct non-defining hyperedges. The following is a Hamiltonian Berge cycle 
\[
v_1,h_1,v_2,h_2,\dots,v_i,f_n,v_n,h_{n-1},v_{v-1},h_{n-2},\dots,v_{i+1},f_1,v_1
\]
a contradiction. Similarly, we have 
\[
N_{\h_1}[S_1]^-\cap N_{\h_n}[S_n]=\emptyset
\]
where $\h_n:=\h\setminus P \cup \{h_{n-1},h_{n-2}\dots,h_{x'}\}$.  Indeed if $v_i\in N_{\h_1}[S_1]^-\cap N_{\h_n}[S_n]$ then $x\leq i\leq x'$, since $\h$ is not Hamiltonian. Thus if the corresponding incidences with $v_i$ or $v_{i+1}$ come from a defining hyperedge $h_j$ of $P$, then either $j<x$ or $j\geq x'$. 
Without loss of generality let us assume $j<x$, then $v_j,h_{j-1},v_{j-1},\dots v_1,f,v_{j+1},\dots,v_n$ is  a Hamiltonian Berge path  such that $h_j$ is not a defining hyperedge and vertices $v_x,v_{x+1}\dots ,v_{x'}$ have the same position as they had in $P$. Thus we may assume both incidences are from non-defining hyperedges and we have a Hamiltonian Berge cycle in $\h$, a contradiction.

Let us denote the number of defining hyperedges incidents with $v_n$ by $k_n$. We assume  $k,k_n\leq \floor{\frac{n+1}{2}}+\mathbbm{1}_{r\geq 4}$. By Claim~\ref{Claim:T_1_size} either $\abs{N_{\h_1}(v_1)}>\floor{\frac{n+1}{2}}$ or we have a vertex $u$, such that $u\in T_1$ and  $u\notin N_{\h_1}[S_1]^-$ and $\abs{N_{\h_1}(v_1)}=\floor{\frac{n-1}{2}}$. 
Even more, $u$ is a terminal vertex of a Hamiltonian Berge path, after altering $P$ with a non-defining hyperedge not incident with a vertex in $S_n$. 
Therefore,  vertex $u$ is not in $N_{\h_1}[S_1]^-\cup N_{\h_n}[S_n]$. Due to the symmetry of a Berge path, we may apply Claim~\ref{Claim:T_1_size} for vertex $v_n$.  
Therefore, either $\abs{N_{\h_n}[S_n]}>\floor{\frac{n+1}{2}}$ or we have a vertex $u'$ such  that $u' \notin N_{\h_1}[S_1]\cup N_{\h_n}[S_n]^+$.

Note that if $n$ is even we may assume that vertices $u$ and $u'$ are not consecutive in $P$ in this given order. Since while applying Claim~\ref{Claim:T_1_size} we may choose a different defining hyperedge $h_y$ for the vertex $v_1$ and $v_n$.
Finally  we have $N_{\h_1}[S_1]^-\cup N_{\h_n}[S_n]\subseteq \{v_1,v_2,\dots,v_{n}\}$,  $N_{\h_1}[S_1]^-\cap N_{\h_n}[S_n]=\emptyset$ and if there are $\{u\}$, $\{u'\}^-$ then they are not contained in $N_{\h_1}[S_1]^-\cup N_{\h_n}[S_n]$. Recall $\abs{N_{\h_1}[S_1]^-}$ and $\abs{N_{\h_n}[S_n]}$ are both at least $\floor{\frac{n+1}{2}}$ by Claim~\ref{claim:neighborhood_size_of_v_1}. 
We have

\[
n+1= 1+\mathbbm{1}_{2|n}+\left( \floor {\frac{n+1}{2}}-1\right) +\floor {\frac{n+1}{2}}\leq \abs{\{u,u'\}\cup N_{\h_1}[S_1]^-\cup N_{\h_n}[S_n]} \leq \abs{\{v_i:i\in [n]\}}=n, 
\]

a contradiction. Note that calculations are the same if $u$ or $u'$ does not exist.
We have analogous contradictions if $\abs{N_{\h_1}[S_1]}$ or $\abs{N_{\h_n}[S_n]}$ is greater than $\floor{\frac{n-1}{2}}$.  

Let $r>3$ and either $k$ or $k_n$ is larger than $\floor{\frac{n-1}{2}}+1$. 
Note that 
\begin{equation}\label{equation:some}
    \{v_i:v_n \in h_i\}\cap N_{\h\setminus P}(v_1)^-=\emptyset
\end{equation}
since otherwise we have a Hamiltonian Berge cycle.
If exactly one is larger than $\floor{\frac{n-1}{2}}+1$, say $k\leq \floor{\frac{n-1}{2}}+1<k_n$ then we have $N_{\h\setminus P}(v_1)\geq \floor{\frac{n-1}{2}}$ and $\{v_i:v_n \in h_i\}\geq \floor{\frac{n-1}{2}}+2$. We have a similar contradiction since $\{v_i:v_n \in h_i\}\cup N_{\h\setminus P}(v_1)^-\subseteq \{1,2,\dots,n-1\}$ and because of Equation~\ref{equation:some}.  Finally if $ \floor{\frac{n-1}{2}}+1<k\leq k_n$, then 
\[
 \binom{\floor{\frac{n-1}{2}}}{r-1}< d(v_1) \leq \binom{\abs{N_{\h\setminus P}(v_1)}}{r-1}+k \leq \binom{n-k_n}{r-1}+k\leq  \binom{n-k}{r-1}+k.
\]
We have a contradiction similarly as before since $\{v_i:v_n \in h_i\}\cap N_{\h\setminus P}(v_1)^-\not=\emptyset$.
Therefore, if $r\geq 4$, we are done. From here we assume $r=3$.

\begin{claim}\label{claim:hyperedge_r=3}
Let $r=3$ then each defining hyperedge is incident with at most one terminal vertex of~$P$.
\end{claim}
\begin{proof}
Every defining hyperedge $h_i$, $1<i<n-1$, is incident with at most one terminal vertex of $P$, since $r=3$. 
If $v_n\in h_1$ or $v_1 \in h_{n-1}$, then there is a Berge cycle containing all vertices but $v_1$ or $v_n$. Let us assume $v_n\in h_1$ thus we have a Berge cycle containing all vertices but $v_1$, namely
\[
v_n,h_1,v_2,h_2,\dots, h_{n-1},v_n
\]
If there are consecutive vertices $v_i,v_{i+1}$ incident with $v_1$ and there is a hyperedge $\{v_i,v_{i+1},v_1\}$ in $\h$ we may assume $h_i=\{v_i,v_{i+1},v_1\}$, hence it is a defining hyperedge. Therefore, the vertices incident with $v_1$ with a non-defining hyperedge are not consecutive on the cycle otherwise we have a Hamiltonian Berge cycle. 
Even more, if a defining hyperedge $h_i$ is incident with $v_1$ then $v_i$, $v_{i+1}$ are not incident with $v_1$ with a non-defining hyperedge. Therefore, the maximum degree of $v_1$ is at most $\binom{\frac{n-1}{2}}{2}+\mathbbm{1}_{2|n}$ a contradiction. Therefore, we have all hyperedge of $\h$ incident with at most one terminal vertex of $P$.
\end{proof}

From the last claim we have $(N_{\h_1}[S_1]\cup \{v_i|v_1\in h_i\})^-\cap (N_{\h_n}[S_n]\cup \{v_{i}|v_n\in h_i\})=\emptyset$ since there is no Hamiltonian Berge cycle.  Also by the same claim we have $k$ or $k_n$ is less than  $\floor{\frac{n-1}{2}}$. If we consider the set $N_{\h_1}[S_1]\cup \{v_i|v_1\in h_i\}$ instead of $N_{\h_1}[S_1]$ and the set $N_{\h_n}[S_n]\cup \{v_{i}|v_n\in h_i\}$ instead of $N_{\h_n}[S_n]$ we get a contradiction in with the same argument as for $r\geq 4$. 
\end{proof}

\subsection{Proof of the theorem for non-uniform hypergraphs}

\begin{proof}[Proof of Theorem~\ref{Theorem_Posa_r_Uniform}]
Here we re-start proof for non-uniform hypergraphs.  For some integer $t$, such that $t>1$, we have
 \[
 2^{t}<d(v_1)\leq  2^{t+1}.
\]

\begin{claim}\label{claim:neighborhood_size_of_v_1_NON-uniform}
 We have $\abs{N_{\h\setminus P}(v_1)}\geq t$.

\end{claim}
\begin{proof}
We may upper bound the number  of $P$ non-defining  hyperedges incident with $v_1$ with $ 2^{\abs{N_{\h\setminus P}(v_1)}}$. Thus we have
\begin{equation}\label{Equation:neighborhood_size_of_v_1_NON-uniform}
 2^{\abs{N_{\h\setminus P}(v_1)}}+k\geq d(v_1)>   2^{t}= 2^{t-1}+ 2^{t-1}.
\end{equation}
For each defining hyperedge incident with $v_1$ there are distinct terminal vertices in $T_1$ hence either $k\leq t$ or $t\geq \floor{\frac{n-1}{2}}$ by Condition~\ref{Condition:PosaNON-uniform_1}. 
If $k\leq t$ then because of Equation~\ref{Equation:neighborhood_size_of_v_1_NON-uniform} we have $2^{\abs{N_{\h\setminus P}(v_1)}}>   2^{t-1}$, since $t>1$ and we are done.

Note that we have $k\leq n-1\leq 2^{\floor{\frac{n-1}{2}}-1}$ for all $n>8$. 
Hence if $t\geq \floor{\frac{n-1}{2}}$ then we are done because of Equation~\ref{Equation:neighborhood_size_of_v_1_NON-uniform}.

\end{proof}

Recall that  $x$ denotes the largest integer for which $\{v_1,v_2,\dots,v_x\}\subseteq N_{\h \setminus P}[v_{1}]$.
\begin{claim}\label{Claim:T_1_size_NON-uniform}
If $\abs{N_{\h_1}[S_1]}=t+1$  then $v_x\in T_1$ or $v_y \in T_1$ for some $y$, where $v_{y+1} \notin N_{\h_1}[S_1] $ and $v_1\in h_y$.

\end{claim}

\begin{proof}
By Claim~\ref{claim:neighborhood_size_of_v_1_NON-uniform}, we have  $\abs{N_{\h\setminus P}(v_1)}\geq t$, hence we have $\abs{N_{\h\setminus P}(v_1)}= t$ and $N_{\h\setminus P}(v_1)=N_{\h_1}[S_1]\setminus \{v_1\}$, since  $N_{\h\setminus P}(v_1)\subseteq N_{\h_1}[S_1]$.  There is a defining hyperedge $h_y$ incident with $v_1$, $y\geq x$, since $d(v_1)> 2^{t}$ and  $\abs{N_{\h_1}(v_1)}=t$. Note that if $t\geq \floor{\frac{n-1}{2}}$ and $n$ is even then there are two such hyperedges. Since we are going to use this Claim for $v_1$ and $v_n$ also we will choose different hyperedges for them. Either $v_{y+1} \notin N_{\h_1}[S_1] $ and we are done by Claim~\ref{Claim:basic}, or $v_{y+1} \in N_{\h_1}[S_1] $, in this case we will show that $v_x \in T_1$. 

If $x=1$, then there is nothing to prove. If $x=2$, then after altering the path $P$ with Permutation~$3$, for a defining hyperedge $h_y$ and non-defining hyperedge $h$ incident with $v_1$ and $v_{y+1}$, $v_2$ will be a terminal vertex of the longest path. 
Hence we have $v_x \in T_1$ if $x$ is two.

If $x \geq 3$, then we have $t\geq x\geq 3$ and we distinguish two cases. At first, we consider $t>3$ and then $t=x=3$.

If $x\geq 3$ and $t>3$, we consider two cases. Note that $v_{x-1}$ and  $v_{y+1}$ are elements of $N_{\h\setminus P}(v_1)$. If there is a non-defining hyperedge incident with vertices $v_{x-1}$ and $v_{y+1}$ then $v_x$ is a terminal vertex after altering $P$ with Permutation~$3$ and we are done. 
Otherwise the number of non-defining hyperedges incident with $v_1$ is at most  $2^t-2^{t-2}$. Note that either  $k\leq t$ or $k>t$. If $k>t$ then  $t\geq  \floor{\frac{n-1}{2}}$ and $k\leq n-1\leq 2^{ \floor{\frac{n-1}{2}}-2}\leq 2^{t-2}$ since $n\geq 13$. If $k\leq t$,  then $k\leq t \leq 2^{t-2}$ since $t>3$. Finally we can upper bound the degree of $v_1$, $ 2^t<d(v_1)\leq k +  2^{t}-  2^{t-2}\leq 2^t,$
a contradiction.

Here we consider the final case when $x=t=3$.  If there is a non-defining hyperedge $h$ incident with $v_2$ and $v_{y+1}$, then $v_3$ is a terminal vertex after altering the path $P$ with Permutation~$3$, for a defining hyperedge $h_y$ and non-defining hyperedge $h$. 
Note that, if $v_2$, $v_{y+1}\in h_1$   we can replace $h_1$ with some a non-defining hyperedge incident with $v_1$ and $v_2$ since $v_2\in N_{\h\setminus P}(v_1)$ and alter the new path with Permutation~$3$, for defining hyperedge $h_y$ and non-defining hyperedge $h_1$ then $v_3$ is a terminal vertex and we are done. Hence we may assume all $P$ non-defining hyperedges and $h_1$  are not incident with both  $v_2$ and $v_{y+1}$ vertices. Recall that since $N_{\h\setminus P}(v_1)=N_{\h_1}[S_1]$ we have  hyperedges $h_1,h_2,\dots, h_x\subseteq N_{\h_1}[v_1]$. The degree of $v_1$ can be upper bounded by the number of non-defining hyperedges which is at most $2^t-2^{t-2}$, even more, the bound includes hyperedge $h_1$. Since $t=3$,  the defining hyperedges incident with $v_1$ is at most three including $h_1$. Hence we have a contradiction
$2^3<d(v_1)\leq (3-1) +  (2^{3}-  2^{3-2})=2^3.$
\end{proof}

From the last claim we have $T_1\geq t+1$, hence since $v_1$ has the maximum degree in $T_1$, we have $t\geq \floor{\frac{n-1}{2}}$ and $d(v_1)> 2^{\floor{\frac{n-1}{2}}}+\mathbbm{1}_{2|n}$ by Condition~\ref{Condition:PosaNON-uniform_1} and~\ref{Condition:PosaNON-uniform_2} of Theorem~\ref{Theorem_Posa_non_Uniform}. 
Note that we can define the set $N_{\h_n}[S_n]$ for the vertex $v_n$ as for $v_1$.  
We have $N_{\h_1}[S_1]^-\cap N_{\h_n}[S_n]=\emptyset$ otherwise we have a Hamiltonian cycle. 
Recall that we have $\abs{N_{\h_1}[S_1]^-}\geq \floor{\frac{n+1}{2}}$, and we have the same inequality for $S_n$, namely $\abs{N_{\h_n}[S_n]}\geq \floor{\frac{n+1}{2}}$ by symmetry of a Berge path.

 

Let $T_1'$ be $N_{\h_1}[S_1]$  if $\abs{N_{\h_1}[S_1]}>\floor{\frac{n+1}{2}}$. If $\abs{N_{\h_1}[S_1]}=\floor{\frac{n+1}{2}}$,  then there is a vertex $u$, such that $u\in T_1$ and $u \notin N_{\h_1}(v_1)^-$ by Claim~\ref{Claim:T_1_size_NON-uniform}. We denote the set $N_{\h_1}[S_1]\cup \{u\}$ by $T_1'$ . Note that $u$ cannot be a terminal vertex after altering $P$ using hyperedges which are not incident with a vertex in $S_n$, otherwise, there is a Hamiltonian Berge cycle. Hence we have $u\notin N_{\h_1}[S_1]^-\cup N_{\h_n}[S_n]$. 
Let $T_n'$ be $N_{\h_n}[S_n]$  if $\abs{N_{\h_n}[S_n]}>\floor{\frac{n+1}{2}}$. If $\abs{N_{\h_n}[S_n]}=\floor{\frac{n+1}{2}}$,  then there is a vertex  $u'$, such that $u'\in T_n$ and $u' \notin N_{\h_n}(v_n)^+$ by Claim~\ref{Claim:T_1_size_NON-uniform}. We denote the set $N_{\h_n}[S_n]\cup \{u'\}$ by $T_n'$. Note that $u'$ cannot be a terminal vertex after altering $P$ using hyperedges that are not incident with a vertex in $S_1$, otherwise, there is a Hamiltonian Berge cycle. 
 Hence we have $\{u'\}^-\nsubseteq N_{\h_1}[S_1]^-\cup N_{\h_n}[S_n]$. 
 Even more, if $n$ is even and both of the vertices $u$ and $u'$ exist then we have $\{u\}\neq\{u'\}^-$.  
 Since we have $N_{\h_1}[S_1]^-\cup N_{\h_n}[S_n]\subseteq \{v_1,v_2,\dots,v_n\} $,  $N_{\h_1}[S_1]^-\cap N_{\h_n}[S_n]=\emptyset$. We have a contradiction 
\[n+1= 2\floor{\frac{n+1}{2}}-1+1+\mathbbm{1}_{2|n}\leq \abs{(T_1')^-\cup T_n')} \leq \abs{\{v_1,v_2,\dots,v_{n}\}}=n. \]

\end{proof}

\section{Concluding remarks}

The next natural step in this direction would be finding the necessary and sufficient conditions for a degree sequence to be $r$-Hamiltonian. In particular,  generalization of Chvátal's Theorem~\ref{thm:chvatal}. We believe that Chvátal's condition would be only sufficient but not necessary for hypergraphs. We pose the following Conjecture. 
\begin{conjecture}
An integer sequence $(d_1,d_2,\ldots,d_n)$ such that $d_1\leq d_2\leq \dots\leq d_n$, $n>2r$ and $r\geq 3$ is  $r$-Hamiltonian if the following conditions hold
\begin{equation}
   d_i>i \mbox{ for } 1 \leq i < r,
\end{equation}

and

\begin{equation}
 \mbox{If } d_i\leq  \binom{i}{r-1}  \mbox{ then } d_{n-i}> \binom{n-i-1}{r-1} \mbox{ for } r \leq i \leq \floor{\frac{n-1}{2}},
\end{equation}

\begin{equation}
 \mbox{If } d_{\frac{n-2}{2}} \leq \binom{\frac{n-2}{2}}{r-1} +1  \mbox{ then } d_{\frac{n+2}{2}}> \binom{\frac{n}{2}-2}{r-1}+(\frac{n}{2}+1)\binom{\frac{n}{2}-2}{r-2}   \mbox{ if $n$ is even.} 
\end{equation}

\end{conjecture} 

At the same time, we find it challenging to pose any conjecture with the necessary and sufficient conditions. 

\section{Acknowledgements}

I would like to thank Ervin Győri for the helpful discussions and their comments on the earlier stages of this work.

We extend our sincerest gratitude to the reviewer for their meticulous reading and invaluable suggestions, which have significantly enhanced the quality of the presentation. We deeply appreciate the time and effort dedicated to reviewing our manuscript, which has undoubtedly contributed to improving this work.

The research is supported by the National Research, Development, and Innovation Office -- NKFIH, grant K132696. 

\bibliography{Proposal.bib}
\end{document}